\documentclass[a4paper,11pt,reqno]{article}
\usepackage[body=16cm,top=3cm,bottom=4cm]{geometry}
 \usepackage{amsrefs}
\usepackage{amsmath,amsfonts,amsthm,amscd,amssymb,mathrsfs}
\usepackage[colorlinks=true]{hyperref}
\usepackage{color}

\newcommand{\eps}{\varepsilon}

\makeatletter
\newcommand{\eqcolon}{\mathrel{\mathord{=}\raise.2\p@\hbox{:}}}
\newcommand{\coloneq}{\mathrel{\raise.2\p@\hbox{:}\mathord{=}}}
\makeatother
\newcommand{\der}{\delta}
\newcommand{\dd}{\mathrm{d}}

\newcommand{\RR}{\mathbb{R}}

\newcommand{\cF}{\mathscr{F}}

\newtheorem{theorem}{Theorem}[section]

\newtheorem{corollary}[theorem]{Corollary}

\newtheorem{definition}[theorem]{Definition}

\newtheorem{lemma}[theorem]{Lemma}

\newtheorem{proposition}[theorem]{Proposition}

\newtheorem{remark}[theorem]{Remark}

\begin{document}
\title{Support theorem for a singular SPDE: the case of gPAM}
\author{K.Chouk, P.K.Friz\\ TU Berlin, TU and WIAS Berlin}

\maketitle
\begin{abstract}
We consider the generalized parabolic Anderson equation (gPAM) in 2 dimensions with periodic boundary. This is an example of a singular semilinear stochastic partial differential equation in the subcritical regime, with (renormalized) solutions only recently understood via Hairer's regularity structures and, in some cases equivalently, paracontrollled distributions by Gubinelli, Imkeller and Perkowski. In the present paper we utilise the paracontrolled machinery and obtain a (Stroock--Varadhan) type support description for the law of gPAM. In the spirit of rough paths, the crucial step is to identify the support of the enhanced noise in a sufficiently fine topology. 
\end{abstract}
\tableofcontents

\section{Introduction}
In a major recent advance, carried out  independently (and with different techniques) by Hairer \cite{hairer_theory_2013} and 
Gubinelli, Imkeller and Perkowski  \cite{gubinelli_paraproducts_2012}  it was understood how to make rigorous sense of a number of important singular semi-linear {\it stochastic partial differential equations} (SPDEs) arising in mathematical physics.  While Hairer's theory of regularity structures can handle more general classes of such highly irregular SPDEs, the two theories yield essentially equivalent results in  a number of interesting cases, including the (generalized) parabolic Anderson model in a {spatial continuum} of dimension $2$, on which this article will focus. More specifically, we consider a
 solution $u:\mathbb R^+\times\mathbb T^2\to\mathbb R$ to the following SPDE (cf. Theorem~\ref{th:Pam-exis} below)
\begin{equation}  \label{eq:Cauchy}
\left\{
\begin{split}
&\mathscr L u=f(u)\xi
\\&u(0,x)=u_0(x)\in\mathcal C^{\alpha}(\mathbb T^2).
\end{split}
\right.
\end{equation}
\noindent Here $\mathscr L=\partial_t-\Delta$ is the heat-operator, $\Delta$ the Laplacian on the two dimensional torus $\mathbb T^2$,  $\mathcal C^{\alpha}(\mathbb T^2)$ is the Besov space $\mathscr B^\alpha_{\infty,\infty}$(see~\eqref{eq:Besov} for the exact definition), $f \in C^3(\mathbb R; \mathbb R)$ a three times differentiable function, and at last $\xi$ is {\it spatial white noise} with (for convenience) {\it zero spatial mean}; that is $\xi$ is a centred Gaussian field with%
\footnote{More precisely, $\xi$ is a centered Gaussian field indexed by $L^2(\mathbb T^2)$ so that $\mathbb E[(\xi,\varphi)(\xi,\psi)]=  (\varphi,\psi) - (\varphi,1)(\psi,1)$.  Note that this covariance structure indeed implies zero (spatial) mean, i.e. $\int_{\mathbb T^2}\xi(x)\dd x=0$ almost surely.}
$$
\mathbb E[\xi(x)\xi(y)]=\der(x-y) - 1.
$$ 


%
%
%

The aim of this paper is to give a characterization for the topological support of the law of the solution $u$ in a suitable H\"older-Besov space.

\vspace{0.6cm}

Even in the well-understood and classical context of { stochastic
differential equations} (SDEs), such a  ``support theorem" is a deep result and was first obtained in a seminal paper by Stroock--Varadhan \cite{StVa72}, many extensions
and alternative proofs followed. A basic observation, used in virtually all this works, is that Wong--Zakai approximations give the ``easy" inclusion in the support theorem. Most relevant for us, Lyons' {\it rough path theory} \cite{L98, LQ02, LCL07} has provided a ``robust" view on SDE theory which subsequently led to decisive  proofs of the Stroock--Varadhan support theory:  the problem is reduced to establish the support characterization for the enhanced noise, in sufficiently strong topologies, upon which the solution depends in a continuous fashion. This strategy of proof was carried our first by Ledoux, Qian and Zhang \cite{LQZ02}, see also Friz, Lyons and Stroock \cite{FLS} and the references in \cite[Ch.19]{FrizVictoir}.

\vspace{0.5cm}
 
The theories of regularity structures and paracontrolled distributions, both inspired by rough path theory, provide an equally ``robust" view on the  classes of SPDEs, which they helped to define in the first place. A similar route towards support characterizations should then be possible. To this end, a number of technical problems need to be overcome, the perhaps most immediate being the divergence of ``Wong--Zakai" approximations\footnote{See Hairer-Pardoux \cite{HP14} for a study of renormalized Wong--Zakai approximations.} due to an infinite It\^o--Stratonovich correction. (Such a problem was already encountered in the classical literature for SPDEs, in particular the work of Bally, Millet and Sanz-Sol\'e~\cite{A.milletM.Sole} appears close to ours in spirit, although of course both techniques and classes of considered SPDEs are entirely different.)

\vspace{0.5cm}

It is clear that our general d\'emarche invites generalizations to other SPDEs where the theory of regularity structures or paracontrolled distributions can be employed, notably the
three dimensional stochastic quantization equation $\Phi_3^4$ ~\cite{hairer_theory_2013, CC13}, the KPZ equation~\cite{hairer_solving_2013} and \cite[Ch.15]{FH14} and its generalizations
presently studied by Bruned, Hairer and Zambotti. Indeed, it would be very desirable, although we believe this is presently out of reach, to have a ``general support theorem" that applies to all (local) solutions to subcritical SPDEs. A more realistic programme 
consists of tackling each singular SPDE of interest with a tailor-made analysis, inspired/extending the one presented in this paper in the case of gPAM \cite{hairer_theory_2013, gubinelli_paraproducts_2012}. For instance, we are convinced that $\Phi_3^4$, with paracontrolled analysis due to  ~\cite{CC13}, can be dealt within a reasonably similar framework even though the details seem to require a considerable additional effort.\footnote{The situation may be compared to the support theorem for fractional Brownian rough paths, 
cf. \cite{FrizVictoir} and the references therein: the case $H \in (1/4,1/3]$, which requires a ``level-3'' enhancement of the noise, is well-known to be substantially more involved than the case $H \in (1/3,1/2]$.}
%
%
%
%
%

%

\vspace{0.6cm}

Having explained our focus on gPAM, we recall the basic issue with this model given by equation \eqref{eq:Cauchy}: the problem is that $\xi$ is too rough to have a well-defined product $f(u)\xi$. Indeed, it is well-known that
 $\xi\in\mathcal C^{-1-\delta}(\mathbb T^2)$ a.s. for all $\delta>0$, and no better, which implies at least formally that at best $u(t,.)\in\mathcal C^{1-\delta}$ a.s., in view of  regularization properties of the heat flow. It is well-known from harmonic analysis that Schwartz distributions in such Besov--H\"older spaces can be 
 multiplied only if the exponents add up to a positive number, which is plainly not the case here and leaves one with the ill-defined product  $f(u) \xi$. 
 If one proceeds by brute force approximation arguments, one quickly finds that the limiting equations is not \eqref{eq:Cauchy} but of the (non-sensical) form
  \begin{equation}
\left\{
\begin{split}
& \mathscr Lu=f(u)\xi -\infty (....)
\\&u(0,x)=u_0(x)\in\mathcal C^{\alpha}(\mathbb T^2).
\end{split}
\right.     \label{eq:NSCauchy}
\end{equation}
 This problem has been treated in  the simple case of $f(u)=u$ in~\cite{Hu} by interpreting the product as the Wick product generated by the Gaussian structure of the white noise and then a chaos expansion of the solution is obtained. For the case of general $f$, no such trick will work.
 
 \vspace{0.6cm} 
 
 As already mentioned, recently two different approaches have been developed to deal with this singular SPDE. One is based on the theory of 
 regularity structure due to M.Hairer~\cite{hairer_theory_2013}, the other one the paracontrolled distribution approach due to  Gubinelli, Imkeller and Perkowski ~\cite{gubinelli_paraproducts_2012}. In the latter, the authors use the Bony paraproduct (see~\cite{BCD-bk,bony}) to obtain a space of distributions which admit some sort of Taylor expansion 
 where in a sense the pointwise product is replaced by the Bony paraproduct, and which is ultimately seen to contain the solution $u$ of the equation, with good properties vis a vis to the afore-mentioned multiplication $f(u)\xi$.
 
 Both theories yield a notion of {\it renormalized} solution for \eqref{eq:Cauchy}, {\it  local in time}, obtained as limit 
 $ u  = \lim_{\eps \to 0} u^\eps$, limit taken in  $C([0,\tau)],\mathcal C^{\alpha})$, $\alpha<1$, with explosion time $\tau = \tau (\omega) >0$ a.s.,
 \begin{equation}
\left\{
\begin{split}
& \mathscr Lu^\eps=f(u^\eps)\xi^\eps-c_\eps f'(u^\eps)f(u^\eps)
\\&u(0,x)=u_0(x)\in\mathcal C^{\alpha}(\mathbb T^2),
\end{split}
\right.     \label{eq:NSCauchy2}
\end{equation}
with $\rho$ a suitable mollifier function, $\xi^\eps=\eps^{-2}\rho(\frac{\cdot}{\eps})\star\xi $ mollified white noise and diverging constants $c_\eps = c_\eps(\rho)$, see Theorem~\ref{th:Pam-exis}. (We insist that the limit $u$ does not depend on the choice of $\rho$.)  Assuming {\it non-explosion} the above convergence takes place in $C([0,T],\mathcal C^{\alpha})$, for any fixed $T>0$. 
We note that non-explosion holds in the linear case $f(u)=u$; a non-explosion condition for non-linear $f$ was recently given in \cite{CFG15x}
and includes the case of compactly supported $f \in C_c^{3}(\mathbb R)$.
%
%
Let us remark that the assumption of non-explosion is not essential for our work, however it removes the need for attaching a cemetery state to the state space (as done, also in the context of a singular SPDE, in \cite{hairer_weber_LDP}).

Let us also emphasize that, $u  = \lim_{\eps \to 0} u^\eps $ should not be considered as the only solution to the (formal) Cauchy problem for \eqref{eq:Cauchy}: given a real constant $a$, replacing $c_\eps$ by $\tilde{c}_\eps \equiv c_\eps + a$ in equation
\eqref{eq:NSCauchy2}  one indeed gets a (in general different) limit $\tilde{u} =  \lim_{\eps \to 0} \tilde{u}^\eps$, which may also be regarded as solution to  \eqref{eq:Cauchy}. 
\newpage
Writing $u[a]  \equiv \tilde{u}$, this is effectively a reflection of the renormalization group, here $a \in (\mathbb{R},+)$, which acts on renormalized solutions of this SPDE, point of view emphasized in \cite{hairer_theory_2013}.

%
%
%
%

 \vspace{0.4cm}

Let $\mathcal H$ be the Cameron-Martin space associated to $\xi$ i.e. the set of $f\in L^2(\mathbb T^2)$ with zero-mean, $\int_{\mathbb T^d}f(x)\dd x=0$. Define the separable space $C^{0,{\alpha}}(\mathbb T^d)$ as the closure of smooth functions in $C^{\alpha}(\mathbb T^d)$. Assuming non-explosion, the law of $u$ can then be regarded as (Borel-)meausre on the (Polish) space $C([0,T],\mathcal C^{0,\alpha}(\mathbb T^2))$. We are now ready to state our main result. 
  
\begin{theorem}\label{th:main-result}
Let $T>0$, $\alpha\in(2/3,1)$, $u_0\in\mathcal C^{0,\alpha}(\mathbb T^d)$ and $f\in C^3(\mathbb R)$. Assuming non-explosition, denote by  $u=u[0]$ the solution 
of the Cauchy problem~\eqref{eq:Cauchy} given by Theorem~\ref{th:Pam-exis} and by $u_\star\mathbb P$ the law of $u$ in $C([0,T],\mathcal C^{0,\alpha}(\mathbb T^2))$. Then we have, with the closure below taken in $C([0,T],\mathcal C^{0,\alpha}(\mathbb T^2))$,
\begin{equation} \label{suppp}
\begin{split}
\text{supp}(u_\star\mathbb P)&=\overline{\left\{\mathscr S(u_0,h,c),\quad h\in \mathcal H, c>0\right\}},
\end{split}
\end{equation}
where $\mathscr S(u_0,h,c)=v$ is the classical solution to (cf. Proposition \ref{prop:class} in the Appendix) 
\begin{equation} \label{equ:5}
\left\{
\begin{split}
& \mathscr Lv=f(v) h-cf'(v)f(v)
\\&v(0,x)=u_0(x)\in\mathcal C^{0,\alpha}(\mathbb T^2).
\end{split}
\right.
\end{equation} 
At last, the support is invariant under the action of the renormalization group in the sense that,
\begin{equation}
\label{invariance}
\text{supp}((u[a])_\star\mathbb P) = \text{supp}(u_\star\mathbb P), \, \, \, \text{for any  } a \in \mathbb{R}.
\end{equation}
\end{theorem}
We remark that the infinite term  in \eqref{eq:NSCauchy} is replaced by a finite expression in \eqref{equ:5}, of the
form $cf'(u)f(u)$, and this is a key aspect in the analysis. It should be noted that, in general, the constant $c$ which appears in (\ref{suppp}) and ranges over all positive reals\footnote{What actually matters is that $c$ ranges over a set which has $+\infty$ as accumulation point.}  cannot be omitted; cf. Lemma ~\ref{lemma:non} (and also Lemma ~\ref{ref:strict}). This is in contrast to the results of Bally, Millet and Sanz-Sol\'e~\cite{A.milletM.Sole} where an infinite constant was effectively set to zero. The underlying reason is that they deal with space-time white noise whereas in PAM case we have purely spatial noise. 

{\bf Acknowledgement:} KC and PKF acknowledge support by the European Research Council under
the European Union's Seventh Framework Programme (FP7/2007-2013) / ERC grant agreement nr. 258237. 
PKF acknowledges support from DFG research unit FOR2402.

\section{Well-posedness result for the parabolic Anderson equation}                
\subsection{Besov spaces and Bony paraproduct}
Before stating the main result of~\cite{gubinelli_paraproducts_2012} about the parabolic Anderson equation let us collect some definition and basics facts about the Besov space.
Let $\chi$ and $\rho$ be a nonnegative smooth radial functions such that
\begin{enumerate}
	\item The support of $\chi$ is contained in a ball and the support of $\rho$ is contained in an annulus;
	\item $\chi(\xi)+\sum_{j\ge0}\rho(2^{-j}\xi)=1 \mathrm{\ for\ all}\ \xi\in\RR^d$;
	\item $\mathrm{supp}(\chi) \cap \mathrm{supp}(\rho(2^{-j}.)) = \emptyset$ for $i \ge 1$ and $\mathrm{supp}(\rho(2^{-i}.)) \cap \mathrm{supp}(\rho(2^{-j}.)) = \emptyset$ when $|i-j| > 1$.
\end{enumerate}
(for the existence of such a function see \cite{BCD-bk}, Proposition 2.10.). Then the Littlewood-Paley blocks are defined by:
$$
\Delta_{-1} u = \cF^{-1}(\chi \cF u)\ \mathrm{and\ for}\ j \ge 0, \Delta_j u = \cF^{-1}(\rho(2^{-j}.)\cF u).$$
Where $\mathscr Ff$ is the Fourier transform of a distribution $f\in\mathscr S'(\mathbb R^d)$.
\\
\\
We define the Besov space of distributions by : 
\begin{equation}\label{eq:Besov}
\mathscr B_{p,q}^{\alpha}=\left\{u\in \mathscr S'(\mathbb R^d); \quad ||u||^q_{\mathscr B_{p,q}^{\alpha}}=\sum_{j\geq-1}2^{jq\alpha}||\Delta_ju||^q_{L^p}<+\infty\right\}.
\end{equation}
In the sequel we will deal extensively with the special case of $\mathcal C^{\alpha}:=\mathscr B_{\infty,\infty}^{\alpha}$ and the Sobolev space $H^\alpha:=\mathscr B_{2,2}^\alpha$ and we write $||u||_{\alpha}=||u||_{\mathscr B_{\infty,\infty}^{\alpha}}$. Let us also introduce the space $\check{H}^{\alpha}(\mathbb R^d)$ (respectively $\check{\mathcal C}^\alpha(\mathbb R^d)$) of distributions $f\in H^\alpha(\mathbb R^d)$(respectively $f\in\mathcal C^\alpha(\mathbb R^d)$) such that $\hat f(0)=0$ equipped with the norm of $H^\alpha(\mathbb R^d)$ (respectively $\mathcal C^\alpha(\mathbb R^d)$) and we remark that $\check{H}^{0}=\mathcal H$.
At some point we will deal with stochastic objects and the trick is to work with Besov spaces with finite indexes and then go back to the space $\mathcal C^\alpha$. For that we have the following useful Besov embedding.
\begin{proposition}[Besov embedding]
\label{proposition:Bes-emb} 
Let $1\leq p_1\leq p_2\leq +\infty$ and $1\leq q_1\leq q_2\leq +\infty$. For all $s\in \mathbb R$ the space $\mathscr B_{p_1,q_1}^{s}$ is continuously embedded in $\mathscr B_{p_2,q_2}^{s-d(\frac{1}{p_1}-\frac{1}{p_2})}$. In particular we have $||u||_{\alpha-\frac{d}{p}}\lesssim||u||_{\mathscr B_{p,p}^{\alpha}}$.
 \end{proposition}   
Taking $f\in\mathcal C^{\alpha}$ and $g\in \mathcal C^{\beta}$ we can formally decompose the product as  
$$
fg=f\prec g+f\circ g+f\succ g
$$
with 
$$
f\prec g=g\succ f=\sum_{j\geq-1}\sum_{i<j-1}\Delta_if\Delta_jg\quad\text{(Paraproduct term)}
$$
and
$$
f\circ g=\sum_{j\geq-1}\sum_{|i-j|\leq 1}\Delta_if\Delta_jg \quad \text{(Resonating term).}
$$
With these notations the following results hold.
\begin{proposition}[Bony estimates \cite{bony}]
\label{proposition:Bony-estim}
Let $\alpha,\beta\in\mathbb R$. Then
\begin{itemize}
\item[(i)] For $f\in L^\infty$ and $g\in \mathcal C^\beta$
$$
||f\prec g||_{\beta}\lesssim||f||_{\infty}||g||_{\beta};
$$ 
\item[(ii)] for $\beta<0$, $f\in\mathcal C^\alpha$ and $g\in\mathcal C^\beta$
$$
||f\succ g||_{\alpha+\beta}\lesssim||f||_{\alpha}||g||_{\beta};
$$  
\item[(iii)] for  $\alpha+\beta>0$ and $f\in\mathcal C^\alpha$ and $g\in\mathcal C^\beta$ 
$$
||f\circ g||_{\alpha+\beta}\lesssim ||f||_{\alpha}||g||_{\beta}.
$$ Moreover if we have that $f\in\mathcal C^\alpha$ and $g\in H^\beta$ with $\alpha+\beta>0$ then 
$$
||f\circ g||_{\alpha+\beta-d/2}\lesssim||f||_{\alpha}||g||_{H^\beta}.
$$ 
\end{itemize}
\end{proposition} 
We finish this section by describing the action of the Fourier multiplier operator on the Besov spaces.
\begin{proposition}[Schauder estimate
]
\label{prop:multiply}
Let $m\in\mathbb R$ and $\psi$ a infinitely differentiable function on $\mathbb R^d-\{0\}$ such that $|D^k\psi(x)|\lesssim|x|^{-m-k}$ for all $k$. Then the following bound  
$$
||\psi(D)f||_{\alpha+m}\lesssim||f||_{\alpha}
$$
for $f\in\mathcal C^\alpha$ with $\psi(D)f=\mathscr F^{-1}(\psi\hat f)$. 
\end{proposition}
\begin{remark}
We note that all the above facts about Besov space can be stated on the Torus $\mathbb T^d$ for detail see~\cite{Tr}.
\end{remark}

\subsection{Convergence of the mollified equation}
Let us now discuss the known (global!) existence and uniqueness results for the gPAM. Similar to the resolution of SDEs via rough path theory, the problem is divided in two parts. 
\begin{itemize}
\item A first part  which is purely analytic in which the PDE driven by smooth $\xi$ is extended to ``rougher" driving noise, with values in a ``bigger space" $\mathscr X^\alpha$. 
\item A second purely stochastic step in which it is shown that the white noise $\xi$ can be enhanced in an element $\mathrm \Xi^{pam}\in\mathscr X^\alpha$
\end{itemize}
Let us write $\mathscr L=\partial_t-\Delta$ for the heat-operator. $\check{\mathcal C}^\infty$ denotes the space of smooth functions, with zero mean, on the torus. We have
\begin{theorem}\label{th:gip,}  \cite{gubinelli_paraproducts_2012, hairer_theory_2013} 
Let $\alpha\in(2/3,1)$, $f\in C_b^3(\mathbb R)$ and 
$$\mathscr S_{c}:\mathcal C^\alpha(\mathbb T^2)\times\check{\mathcal C}^\infty(\mathbb T^2)\times\mathbb R\mapsto C(\mathbb R^+,\mathcal C^\alpha(\mathbb T^2))$$ the solution operator  as function of ($u_0, \theta, c)$, for the classical PDE (cf. Proposition \ref{prop:class}) 
\begin{equation}
\left\{
\begin{split}
&\mathscr L u=f(u)\theta-cf'(u)f(u)
\\&u(0,x)=u_0(x).
\end{split}
\right.
\end{equation}
Then, assuming non-explosion,\footnote{Cf. Remark 1.13 in \cite{hairer_theory_2013} or \cite{gubinelli_paraproducts_2012}. We note again that this assumption is not essential but simplifies the presentation removing the need to attaching a cemetery state, sufficient conditions for non-explosion were given in \cite{CFG15x}.} there exist a Polish space $\mathscr X^\alpha$ and a continuous map $$\mathscr S_{r}:\mathcal C^\alpha(\mathbb T^2) \times\mathscr X^\alpha\to\mathcal C(\mathbb R^+;\mathcal C^\alpha(\mathbb T^2))$$ which extends $\mathscr S_{c}$ in the following sense  
\begin{equation}\label{eq:extensionCi}
\mathscr S_{c}(u_0,\theta,c)=\mathscr S_{r}(u_0,\mathscr M(\theta,c))
\end{equation}
with 
\begin{equation}\label{eq:mathscrM}
\mathscr M(\theta,c): =(\theta;\theta\circ K \theta-c)
\end{equation}
 where
$K \theta := (-\Delta)^{-1} \theta$ is the (unique) smooth, zero-mean solution to  $(-\Delta)u = \theta \in \check{\mathcal C}^\infty$, cf. Proposition \ref{prop:laplacian-inverse}.
\end{theorem} %
As is easy to see, each $\mathcal  C^\alpha(\mathbb T^2)$ in the above statement can be replaced by the (separable, hence Polish) $$\mathcal C^{0,\alpha}(\mathbb T^2) = \overline{\left\{\mathcal C^\infty(\mathbb T^2)  \right\}}^{\mathscr C^\alpha}.$$

Let us now be more precise about the ``enhanced noise space" $\mathscr X^\alpha$. %
\begin{definition}
\label{def:rd}
Let $\mathscr H^\alpha:=\mathcal C^{\alpha-2}(\mathbb T^2)\times\mathcal C^{2\alpha-2}(\mathbb T^2)$ and $||F||_{\mathscr H^\alpha}$ denote the norm in this Banach space. Now we define the set $\mathscr X^\alpha$ by the following identity : 
$$
\mathscr X^\alpha:=\overline{\left\{(\theta,\theta\circ K\theta-c);\quad\theta\in\check{\mathcal C}^\infty(\mathbb T^2 )\right),c\in\mathbb R  \}}^{\mathscr H^\alpha}
$$  
Finally we denote by $\Xi=(\Xi^1,\Xi^2)$ a generic element in $\mathscr X^\alpha$. Whenever $\Xi_1=\xi$, we call $\Xi$ an enhancement (or lift) of $\xi$.
\end{definition}

We have the following alternative description of $\mathscr X^\alpha$. Recall  that $\mathcal H = \check L^2$, the space of zero-mean square-integrable functions on the torus, is precisely the Cameron--Martin space for our spatial zero-mean white-noise $\xi$.

\begin{lemma} \label{Xalpalt}
For $\alpha < 1$, the following set identity holds,
$$
 \mathscr X^{\alpha}=\overline{\left\{(\theta,\theta\circ K\theta-c);\quad\theta\in \mathcal H,c\in\mathbb R\right\}}^{\mathscr H^\alpha}.
 $$
\end{lemma}
\begin{proof}
Let $\theta\in \mathcal H$ then by the Besov embedding we have that $||\theta||_{\mathcal C^{\alpha-2}}\lesssim||\theta||_{\mathcal C^{-1}(\mathbb T^2)}\lesssim||\theta||_{L^2(\mathbb T^2)}$. Moreover by a direct computation we get :
 \begin{equation}
 \begin{split}
 ||\theta\circ K\theta||^2_{H^\gamma}&=\sum_{k\in\mathbb Z^2}|k|^{2\gamma}\left|\sum_{k_1+k_2=k;k_2,k_1\ne0;|i-j|\leq1}\frac{1}{|k_2|^2}\mathscr F(\Delta_i\theta)(k_1)\mathscr F(\Delta_j \theta)(k_2)\right|^2
\\&\lesssim\sum_{k\in\mathbb Z^2}|k|^{2\gamma}\left|\sum_{k_1+k_2=k,|k|\lesssim|k_1|\sim|k_2|}\frac{1}{|k_2|^2}\hat\theta(k_1)\hat\theta(k_2)\right|^2\lesssim||\theta||^4_{L^2(\mathbb T^2)}\sum_{k\in\mathbb Z^2}|k|^{2\gamma-4}<+\infty
\end{split}
 \end{equation}
 if $\gamma<1$. 
 Now using the Besov embedding once again we get that $||\theta\circ K\theta||_{\gamma-1}\lesssim||\theta||^2_{L^2(\mathbb T^2)}$  for all $\gamma< 1$ and in particular if we take $2\alpha-2\leq\gamma-1<0$ 
 we get that $||\theta\circ K\theta||_{2\alpha-2}\lesssim||\theta||^2_{L^2(\mathbb T^2)}$ then if we take $\theta^\eps$ a regularization of $\theta$ such that $||\theta^\eps-\theta||_{L^2(\mathbb T^2)}\to^{\eps\to0}0$ we obtain immediately by the same computation that  
\begin{equation*}
||\theta^\eps\circ K\theta^\eps-\theta\circ K\theta||_{0}\lesssim||(\theta-\theta^\eps)\circ K\theta||_{0}+||\theta\circ (-\Delta)^{-1}(\theta-\theta^\eps)||_{0}\lesssim||\theta||_{L^2(\mathbb T^2)}||\theta-\theta^\eps||_{L^2(\mathbb T^2)}
\end{equation*}
and then we obtain the convergence of $(\theta^\eps,\theta^\eps\circ K\theta^\eps)$ to $(\theta,\theta\circ K\theta)$ in $\mathscr H^\alpha$ and this for every $\theta\in \mathcal H$.  In conclusion any element of $\mathcal H$ can be lifted in an rough distribution in $\mathscr X^\alpha$, in other word the following identity :
 $$
 \mathscr X^{\alpha}=\overline{\left\{(\theta,\theta\circ K\theta-c);\quad\theta\in \mathcal H,c\in\mathbb R\right\}}^{\mathscr H^\alpha}
 $$
 hold.
 \end{proof}
 \begin{remark} The extension property \eqref{eq:extensionCi} was given for all smooth zero-mean function on 
 the torus. But it extends to all elements of $\mathcal H$ and this can be seen as follows. 
  
Let $\theta\in \mathcal H$ and $\theta^\eps$ a regularization of $\theta$  such that $||\theta^\eps-\theta||_{L^2(\mathbb T^2)}\to0$ then due to the previous lemma we know that $(\theta^\eps,\theta^\eps\circ K\theta^\eps-c)$ converge to $(\theta,\theta\circ K\theta-c)$ in $\mathscr H^\alpha$ and this gives the convergence of $\mathscr S_r(u_0,(\theta^\eps,\theta^\eps\circ K\theta^\eps-c))$ to $\mathscr S_r(u_0,(\theta,\theta\circ K\theta-c))$ in $C(\mathbb R^+,\mathscr C^\alpha(\mathbb T^2))$. Now taking the classical solution $\mathscr S_c(u_0,\theta^\eps,c)$ to
$$
\mathscr Lu^{\theta^\eps}=f(u)\theta^\eps-cf'(u^{\theta^\eps})f(u^{\theta^\eps}),\quad u(0,x)=u_0(x),
$$
we know that by definition it satisfies the relation
$$
\mathscr S_r(u_0,(\theta^\eps,\theta^\eps\circ K\theta^\eps-c))=\mathscr S_c(u_0,\theta^\eps,c).
$$
And then taking the limit in this equation we obtain immediately that
$$
\lim_{\eps\to0}\mathscr S_c(u_0,\theta^\eps,c)=\mathscr S_r(u_0,(\theta,\theta\circ K\theta-c)).
$$   
Moreover we know by the Proposition~\ref{prop:class} that the map $\theta\mapsto S_c(u_0,\theta,c)$ is continuous from $\mathcal H$ to $C(\mathbb R^+,L^2(\mathbb T^2))$, from which we get the following relation 
$$
\mathscr S_c(u_0,\theta,c)=\mathscr S_r(u_0,(\theta,\theta\circ K\theta-c))
$$
for all $\theta\in \mathcal H$ and $c\in\mathbb R$.
\end{remark}
Recall from the introduction that $\xi$ denotes zero mean spatial white noise on the two-dimensional torus. We consider 
a mollification of this noise. Let $\psi$ be a radial bounded function with compact support which is continuous at the origin, with $\psi(0)=1$, and set 
\begin{equation}
\label{eq:mollification}
\xi^\eps := \xi^\eps(\psi) :=\sum_{k\ne0}\psi(\eps k)\hat \xi(k)e_k
\end{equation}
where $(e_k)$ is the Fourier basis of $L^2(\mathbb T^2)$ then at this point we have the following convergence result.
\begin{theorem}  \cite{gubinelli_paraproducts_2012}  \label{th:Pam-exis}
Let $\alpha<1$. Then, with (diverging!) constants $c_\eps=c_\eps(\psi)\in\mathbb R$ given by  
$$
c_{\eps}=\sum_{k\ne0}\frac{|\psi(\eps k)|^2}{|k|^2}
$$
we have convergence of $\mathscr M(\xi^\eps,c_\eps) \equiv (\xi^\eps;\xi^\eps\circ K \xi^\eps-c_\eps)$ to some limit $\Xi^{pam}$. More precisely, with convergence in $L^p(\Omega,\mathscr X^\alpha)$ for all $p>1$ and almost surely, 
 $$
 \Xi^\eps := \mathscr M(\xi^\eps,c_\eps) \to_{\eps \to 0 } \Xi^{pam} \in\mathscr X^\alpha
 $$ 
 such that $(\Xi^{pam})^1=\xi$. Moreover, the limit $\Xi^{pam}$ does not depend on the function $\psi$ used to mollify the noise. 
   \end{theorem} 
Now, following  \cite{gubinelli_paraproducts_2012}, the point is that with $u^\eps:=\mathscr S_{r}(u_0,\Xi^\eps)$, due to the constants $c_\eps$, the function $u^\eps$ does not satisfy equation \eqref{eq:Cauchy} but a modified equation given by 
$$
\mathscr Lu^\eps=f(u^\eps)\xi^\eps-c_\eps f(u^\eps)f'(u^\eps).
$$  
One then dedcues that $u^\eps$ converges to $u=\mathscr S_{r}(u_0,\Xi)$ in $C([0,T],\mathcal C^\alpha(\mathbb T^2))$ where the convergence is
in $L^p(\Omega,\mathscr X^\alpha)$ for all $p>1$ and almost surely.

\section{Support theorem}
\subsection{Description of the strategy and support theorem for the white noise}
Now to obtain the support theorem for our equation we begin by obtaining the result for the rough distribution $\Xi^{pam}$ associated to the white noise and then we transfer our result to $u$ by using the continuity of the map $\mathscr S_{r}$. Recall that $\mathcal C^{0,\beta}(\mathbb T^2)$ is the ($\beta$-Besov-H\"older) closure of smooth functions on the torus;  
$\check{\mathcal C}^{0,\beta}(\mathbb T^2)$ denotes the closure of zero-mean, smooth functions. As a warm-up, we now characterize the support of the white noise $\xi$ in the Besov-H\"older space.
\begin{proposition} \label{prop:whitenoisesupport}
Let $(\Omega,\mathscr A,\mathbb P)$ the abstract probability space associated to zero-mean white noise $\xi$ on $\mathbb T^2$ and $\xi_\star\mathbb P$ the law of $\xi$ viewed as Borel measure on $\mathcal C^{0,\alpha-2}(\mathbb T^2)$, any $\alpha <1$. Then
$$
\text{supp}(\xi_\star\mathbb P)= \check{\mathcal C}^{0,\alpha-2}(\mathbb T^2).
$$
\end{proposition}
\noindent This follows of course immediately from general facts of Gaussian measures on separable Banach spaces: support equals closure of the Cameron-Martin space ${\mathcal H}$, here given by zero-mean elements in $L^2( \mathbb T^2)$, and it is a simple exercise to verify $\check{\mathcal C}^{0,\alpha-2}(\mathbb T^2) = \overline{\mathcal H}$ with $(\alpha-2)$-Besov-H\"older-closure, any $\alpha <1$. That said, we now (re)proof Proposition \ref{prop:whitenoisesupport} with an argument that 
extends to enhanced noise, as discussed below, which is clearly a non-Gaussian object.
%
%
%
%

The easy half of Proposition \ref{prop:whitenoisesupport}, $\text{supp}(\xi_\star\mathbb P)\subseteq\check{\mathcal C}^{0,\alpha-2}(\mathbb T^2)$, follows at once from the convergence $\xi^\eps \to \xi$ in $\mathcal C^{\alpha-2}$ for all $\alpha<1$, with mollified noise $\xi^\eps (\omega) \in \check{\mathcal C}^{\infty}(\mathbb T^2)$ as introducted in \eqref{eq:mollification}. 
Now to prove the other inclusion let us introduce the translation operator $T_h:\mathscr S'(\mathbb T^2)\to\mathscr S'(\mathbb T^2)$ for $h\in \mathcal H$ defined by $T_h\psi :=h+\psi$. It is immediate to check that $T_h$ is a continuous invertible operator from $\check{\mathcal C}^{0,\alpha-2}$ to $\check{\mathcal C}^{0,\alpha-2}$ with inverse $T_{-h}$.  We then state the (well-known) Cameron--Martin theorem. 
\begin{theorem}
For $h\in \mathcal H$, the law of $\xi$ and the law of $T_h\xi$ are equivalent.
\end{theorem}

A simple consequence of this theorem is that the support of the law of $\xi$ is invariant by $T_h$. Although straight-forward, we spell out the proof for later inspection. 
\begin{lemma}
Let $h\in \mathcal H$ then $T_h\text{supp}(\xi_\star\mathbb P)\subseteq\text{supp}(\xi_\star\mathbb P)$.
\end{lemma} 
\begin{proof} (Lemma) Let $x\in\text{supp}(\xi_\star\mathbb P)$, then - by definition - we know that for any open set $U$ of $\check{\mathcal C}^{\alpha-2}(\mathbb T^2)$ such that $x\in U$ we have 
$\mathbb P(\xi\in U)>0$. Let $V$ be an open set such that $T_hx\in V$. By continuity of $T_h$ we know that there exist an open set $U$ such that $x\in U$ and $T_hU$ is contained in $V$ so that
$
\mathbb P(\xi\in V)\geq\mathbb P(T_{-h}\xi\in U) > 0 
$
where the final  strict positivity follows from $P(\xi\in U)>0$ and the Cameron--Martin theorem. As a consequence, $T_hx\in\text{supp}(\xi_\star\mathbb P)$. 
\end{proof}
We now proof the other half of Proposition \ref{prop:whitenoisesupport}, that is $\text{supp}(\xi_\star\mathbb P)\supseteq\check{\mathcal C}^{0,\alpha-2}(\mathbb T^2)$. Take $x\in\text{supp}(\xi_\star\mathbb P)$. From the first inclusion, it is clear that there exist $(x^n)_n$ a sequence of elements in $\mathcal H$  such that $\lim _{n\to+\infty}x^n=x$ in the space $\check{\mathcal C}^{\alpha-2}(\mathbb T^2)$, equivalently $T_{-x^n}x\to^{n+\infty}0$. By the invariance of the support  of $\xi^\star\mathbb P$ under the translation operator, $T_{-x^n}x\in\text{supp}(\xi_\star\mathbb P)$ and using the fact the support is a closed set in $\mathcal C^{\alpha-2}(\mathbb T^2)$ we obtain immediately that $0\in\text{supp}(\xi_\star\mathbb P)$. Then any $T_h 0 = h$ is also in the support, hence $\mathcal H\subseteq\text{supp}(\xi_\star\mathbb P)$ which gives the second inclusion. 

\subsection{Support theorem for the enhanced white noise}   
The goal of this section is to characterize the support of the law of $\Xi^{pam}$. We have
\begin{theorem}\label{th:supp-pam}
Let $\alpha\in (2/3,1)$ and let $(\Xi^{pam})_\star\mathbb P$ the law of $\Xi^{pam}$ viewed as a Borel measure on $\mathscr X^\alpha$. 
Then\footnote{The second equality was already established in Lemma \ref{Xalpalt}.}
$$
\text{supp}((\Xi^{pam})_\star\mathbb P)=  \mathscr X^\alpha = \overline{\left\{(\theta,\theta\circ K\theta-c):\quad \theta\in \mathcal H,c\in\mathbb R\right\}}^{\mathscr H^\alpha}.
$$
\end{theorem}
\begin{remark} 
Recently it was proved in~\cite{CR} that $\mathscr X^\alpha=\check{\mathcal C}^{0,\alpha-2}(\mathbb T^2)
\times \check{\mathcal C}^{0,2\alpha-2}(\mathbb T^2) $
\end{remark}  
Let us here recall 
\begin{equation} \label{approxim}
\Xi^{pam}=\lim_{\eps\to0}\mathscr M(\xi^\eps,c^\eps) \in \mathscr X^\alpha \text{ a.s.}
\end{equation}
with convergence in the space $\mathscr H^\alpha$, and (Polish) space $\mathscr X^\alpha$ as introduced in Definition~\ref{def:rd}, where $c^\eps=\mathbb E[\xi^\eps\circ K\xi^\eps]=\sum_{k\ne0}|\psi(\eps k)|^2|k|^{-2}$. As a trivial consequence, 
we note that the support of the law of 
$$
\Xi^{pam}[a] :=\lim_{\eps\to 0}\mathscr M(\xi^\eps,c^\eps+a)=\Xi^{pam}+(0,-a)
$$
is not affected by the perturbation $(0,-a)$, i.e.
$$
\text{supp}((\Xi^{pam}[a])_\star\mathbb P)=\overline{\left\{(\theta,\theta\circ K\theta-c):\quad \theta\in \mathcal H,c\in\mathbb R\right\}}^{\mathscr H^\alpha}.
$$
From the equation~\eqref{approxim} we get immediately the easy half of Theorem \ref{th:supp-pam},
$$
\text{supp}((\Xi^{pam})_\star\mathbb P)\subseteq\mathscr X^\alpha.
$$
To obtain the other inclusion we will need the following lemma which can be seen in the context of rough paths, 
 as analogue of highly oscillatory approximations to the so called pure area rough path 
(will also be used in Lemma \ref{lemma:con-cons} below.) 
\begin{lemma}\label{lemma:pure-area}
Let $c\geq0$ then there exist $(X^{n,c})_{n\in\mathbb N}$, smooth functions on the $\mathbb{T}^2$, such that
\begin{enumerate}
\item $||X^{n,c}||_{\alpha-2}\to^{N\to+\infty}0$
\item $||X^{n,c}\circ KX^{n,c}-c||_{2\alpha-2}\to^{N\to+\infty}0$
\end{enumerate}
for $\alpha<1$. In fact, with $z:=(1,1)$ we may take 
\begin{equation*}
X^{n,c}(x)=c^{1/2}2^{n+1}\cos(2^{n}\langle z,x\rangle).
\end{equation*}
\end{lemma}
\begin{proof}
Let $Y^n(x)=2^{n}e^{i\langle 2^{n}z,x\rangle}$ . Then we see that by a simple computation that 
$$
\Delta_qY^n(x)=\rho(2^{-q+n}z)2^{n}e^{i\langle 2^{n}z,x\rangle},\quad\Delta_{-1}Y^n(x)=\chi(2^{n}z)2^{n}e^{i\langle 2^{n}z,x\rangle}
$$
for all $q\geq0$ and then we get easily that 
$$
||Y^n||_{\alpha-2}\lesssim\max(2^{-n(1-\alpha)},2^n\chi(2^nz))\to^{n\to+\infty}0
$$
And by a direct computation we see also that 
$$
Y^n\circ KY^n=\frac{1}{2^{2n}|z|^2}Y^n\circ Y^n=|z|^{-2}e^{i2^{n+1}\langle z,x\rangle}
$$
and then we get that 
$$
||Y^n\circ KY^n||_{2\alpha-2}\lesssim\max(2^{-n(2-2\alpha)},\chi(2^n|z|))\to^{n\to+\infty}0
$$
Now when $c>0$ we can take $X^{n,c}(x)=c^{1/2}(Y^n(x)+\overline {Y^n(x)})=c^{1/2}2^{n+1}\cos(2^{n}\langle z,x\rangle)$ where $\overline {Y^n}$ is the complex conjugate of $Y^n$ and then obviously we have that :
$$
||X^{n,c}||_{\alpha-2}\to^{n+\to\infty}0
$$
moreover we have the following equality
$$
X^{n,c}\circ  KX^{n,c}=2\mathscr Re(Y^n\circ KY^n)+2Y^n\circ \overline{KY^n}=2\mathscr Re(Y^n\circ KY^n)+c
$$
with $\mathscr Re(Y^n\circ KY^n)$ is the real part of $Y^n\circ KY^n$. And then we obtain immediately that 
$$
||X^{n,c}\circ  KX^{n,c}-c||_{2\alpha-2}\lesssim ||Y^n\circ KY^n||_{2\alpha-2}
$$
\end{proof}
And then we obtain immediately the following result, to be compared with Lemma \ref{Xalpalt}.

\begin{proposition}
Let $\alpha\in(2/3,1)$ then we have that 
$$
\mathscr X^{\alpha}=\overline{\left\{(\theta,\theta\circ K\theta-c):\quad \theta\in \mathcal H,c>0\right\}}^{\mathscr H^\alpha}
$$
\end{proposition}
\begin{proof}
Let $\theta\in \mathcal H$, $a\in\mathbb R$, $c>\max(0,a)$ and take $X^{n,c-a}$ as in the Lemma~\ref{lemma:pure-area} and define $\theta^n=\theta+X^{n,c-a}$ then of course we have that 
$$
||\theta^n-\theta||_{2\alpha-2}=||X^{n,c-a}||_{2\alpha-2}\to^{n\to+\infty}0
$$
and a quick computation gives 
$$
\theta^n\circ K\theta^n=\theta\circ K\theta+X^{n,c-a}\circ K\theta+\theta\circ  KX^{n,c-a}+X^{n,c-a}\circ  KX^{n,c-a}
$$
And using the Bony estimate for the resonating term we get that 
\begin{equation*}
\begin{split}
||\theta\circ  KX^{n,c-a}||_{2\alpha-2}&\lesssim||\theta\circ  KX^{n,c-a}||_{\alpha-1}
\\&\lesssim||\theta||_{L^2(\mathbb T^2)}|| KX^{n,c-a}||_{\alpha}\to^{n\to+\infty}0
\end{split}
\end{equation*}
and by the same way we show that $||K\theta\circ X^{n,c-a}||_{2\alpha-2}$ vanish when $n$ go to the infinity. Then we have shown that $(\theta^n,\theta^n\circ K\theta^n-c)$ converge to $(\theta,\theta\circ K\theta-a)$ which gives : 
$$
(\theta,\theta\circ K\theta-a)\in\overline{\left\{(h,h\circ Kh-c):\quad h\in \mathcal H,c>0\right\}}^{\mathscr H^\alpha}
$$
and finally we get 
$$
\mathscr X^\alpha\subseteq\overline{\left\{(\theta,\theta\circ K\theta-c):\quad \theta\in \mathcal H, c>0\right\}}^{\mathscr H^\alpha}
$$
of course the other inclusion is an obvious fact.
\end{proof}
We get from this result that
$$
\text{supp}((\Xi^{pam})_\star\mathbb P)\subseteq\overline{\left\{(\theta,\theta\circ K\theta-c):\quad \theta\in \mathcal H,c>0\right\}}^{\mathscr H^\alpha}=\mathscr X^\alpha.
$$
Now let us focus on the other inclusion. As in the case of the white noise we will need to introduce an appropriate translation operator on the space $\mathscr X^{\alpha}$. Let $h\in\mathcal H$ and we define for $\Xi=(\Xi^1,\Xi^2)\in\mathscr H^\alpha$ the following translation operator :
$$
T_{h}\Xi=(\Xi^1+h,\Xi^2+h\circ  Kh+h\circ K\Xi^1+\Xi^1\circ Kh).
$$
Is not difficult to see that $T_{h}$  is a continuous invertible map on $\mathscr H^\alpha$, the inverse given by  $T_{-h}$. More precisely, we have
%
\begin{proposition}
Let $h\in \mathcal H$. Then we have
$$
||T_{h}\Xi_1-T_{h}\Xi_2||_{\mathscr H^\alpha}\leq 2(||h||_{\mathcal H}+1)||\Xi^1_1-\Xi^1_2||_{\alpha-2}+||\Xi^2_1-\Xi^2_2||_{2\alpha-2}.
$$ 
\end{proposition}
\begin{proof}
By definition we have that : 
\begin{equation*}
\begin{split}
||T_{h}\Xi_1-T_{h}\Xi_2||_{\mathscr H^\alpha}&\leq ||\Xi_1^1-\Xi_2^1||_{\alpha-2}+||\Xi_1^2-\Xi_2^2||_{2\alpha-2}+||h\circ (K\Xi_1^1-K\Xi^1_2)||_{2\alpha-2}
\\&+||(\Xi_1^1-\Xi_2^1)\circ K h||_{2\alpha-2}.
\end{split}
\end{equation*} 
Using the Bony estimates for the resonating term (see~\eqref{proposition:Bony-estim}) and the Schauder estimate for $K$(Proposition~\ref{prop:laplacian-inverse}) we obtain that : 
$$
||h\circ \left(K\Xi^1_1-K\Xi^1_2\right)||_{2\alpha-2}\lesssim||h\circ \left(K\Xi^1_1-K\Xi_2^1\right)||_{\alpha-1}\lesssim||\Xi_1^1-\Xi_2^1||_{\alpha-2}||h||_{\mathcal H}.
$$
Now by the same argument we get :
$$
||Kh\circ(\Xi^1_1-\Xi^1_2)||_{2\alpha-2}\lesssim||\Xi^1_1-\Xi^1_2||_{\alpha-2}||h||_{\mathcal H}
$$
which completes the proof.
\end{proof}
Now we have the following proposition which is the equivalent of the Cameron--Martin theorem for $\Xi^{pam}$.
\begin{proposition}
We have
$$
\mathbb P\left(\left\{\omega\in\Omega;\quad T_h\Xi^{pam}(\omega)=\Xi^{pam}(\omega+h)\quad\text{ for all }h\in\mathcal H\right\}\right)=1.
$$
As a consequence (of the standard Cameron--Martin theorem for Gaussian measures) the laws of  $\Xi^{pam}$ and $T_h\Xi^{pam}$ are equivalent.
\end{proposition}
\begin{proof}
Without loss of generality we can assume that $\Omega=\mathscr S'(\mathbb T^2)$ and $\mathbb P$ is the law of the white noise with zero mean and that $\xi$ is given by the projection process (ie: for $\omega\in\Omega,\quad\xi(\omega)(\phi)=\omega(\phi)$ for all $\phi\in\mathscr S(\mathbb T^2)$). Let us now define $\Xi^{pam,\eps}$ by   
$$
\Xi^{pam,\eps}(\omega):=(\omega^\eps,\omega^\eps\circ K\omega^\eps-c_\eps)
$$
with $\omega^\eps:=\sum_{k}\psi(\eps k)\hat\omega(k)e_k$ then we know that there exist a measurable set $\mathbb A$ with $\mathbb P(\mathbb A)=1$ and such that for all $\omega\in\mathbb A$
the convergence $\Xi^{pam,\eps}(\omega)\to\Xi^{pam}(\omega)$ holds in $\mathscr H^\alpha$. Now taking $\omega\in\mathbb A$ using the fact that 
$$
\Xi^{pam,\eps}(\omega+h)=T_{h^\eps}\Xi^{pam}(\omega)
$$
and the continuity of the translation operator $(h,\Xi)\mapsto T_h\Xi$ we see that $\Xi^{pam}(\omega+\eps)$ is also convergent to $T_h\Xi^{pam}(\omega)$ on the other hand we have that 
$$
\Xi^{pam}(\omega+h)=\lim_{\eps}\Xi^{pam,\eps}(\omega+h)
$$
thanks to the fact that the limit in the r.h.s exist and the fact that very realization of $\Xi^{pam}$ is the limit of $\Xi^{pam,\eps}$. This of course allow us to identify 
$$
T_h\Xi^{pam}(\omega)=\Xi^{pam}(\omega+h).
$$
\end{proof}
Now as in the case of the white noise this last result allows to get the invariance of the support by $T_h$. Indeed we have the following corollary
\begin{corollary}\label{cor:inv-tran-1}
Let $h\in \mathcal H$ and take $\Xi\in\text{supp}((\Xi^{pam})_\star\mathbb P)$ then $T_h\Xi^{pam}\in \text{supp}((\Xi^{pam})_\star\mathbb P)$.
\end{corollary} 
Now to proceed as in the white noise case we will show that the support of $\Xi^{pam}$ contain the $0$ element. This exactly the propose of the next proposition : 
\begin{proposition}\label{prop:supp-point}
Given $a\in\mathbb R$ then there exist $\Xi\in\text{supp}((\Xi^{pam})_\star\mathbb P)$ and $h_a^k\in \mathcal H$ such that 
$$
T_{-h_a^{k}}\Xi\to^{\eps\to0}(0,-a) \quad\text{in } \mathscr H^\alpha.
$$
\end{proposition}
To proof this proposition we will need the following preliminary result :
\begin{lemma}\label{lemma:mixed-conv}
Let $b_\eps$ defined by 
$$
b_{\eps}:=\sum_{k\in \mathbb Z^2,k\ne0}\frac{|\psi(\eps k)|}{|k|^2}
$$
then the following convergence 
$$
\xi\circ K\xi^\eps-b_\eps\to^{\eps\to0}\Xi^{2,pam},\quad\xi^\eps\circ K\xi-b_\eps\to^{\eps\to0}\Xi^{2,pam}
$$
hold in the space $L^p(\Omega, \mathcal C^{2\alpha-2}(\mathbb T^2))$ for every $p>1$.
\end{lemma}
\begin{proof}
We have by a direct computation that :
$$
b_\eps=\sum_{|i-j|\leq1,k\in\mathbb Z^2;k_{12}=k}\rho(2^{-i}k_1)\rho(2^{-j}k_2)\mathbb E[\hat\xi(k_1)\hat\xi(k_2)]\psi(\eps k_1)|k_2|^{-2}e_k=\sum_{k_1}\psi(\eps k_1)|k_1|^{-2}
$$
where we wrote $k_{12}$ instead of $k_1+k_2$ for shorter notation. In view of the convergence giving in the Theorem~\ref{th:Pam-exis} it suffice to prove that :
$$
(\xi^\eps\circ K\xi-b_\eps)-(\xi^\eps\circ K\xi^\eps-c_\eps)\to^{\eps\to0}0
$$
in $L^p(\Omega,\mathcal C^{-\delta}(\mathbb T^2))$. A quick computation gives  
\begin{equation}
\begin{split}
&\left|\Delta_q((\xi^\eps\circ K\xi-b_\eps)-(\xi^\eps\circ K\xi^\eps-c_\eps))(x)\right|^2
\\&=\sum_{\substack{|i_1-i_2|,|j_1-j_2|\leq1\\k_{12}=k,k'_{12}=k'}}\rho(2^{-q}k)\rho(2^{-q}k')\Pi_{l=1}^{2}(\rho(2^{-i_l}k_l)\rho(2^{-j_l}k'_l))
 \\&\times \psi(\eps k)_1\psi(\eps k_2)(1-\psi(\eps k_2))(1-\psi(\eps k'_2))
\\&\times(\hat \xi(k_1)\hat\xi(k_2)-\mathbb E[\hat\xi(k_1)\hat\xi(k_2)])(\overline{\hat \xi(k'_1)\hat\xi(k'_2)-\mathbb E[\hat\xi(k'_1)\hat\xi(k'_2)]})e_{k-k'}(x).
\end{split}
\end{equation}
Then using the Wick theorem we obtain that : 
$$
\mathbb E[|\Delta_q((\xi^\eps\circ K\xi-b_\eps)-(\xi^\eps\circ K\xi^\eps-c_\eps))(x)|^2]=J^\eps_1+J^\eps_2
$$
with 
$$
J^\eps_1=\sum_{\substack{q\lesssim i_1\sim i_2\sim j_1\sim j_2\\k\in\mathbb Z^2,k_{12=k}}}|\rho(2^{-q}k)|^2\Pi_{l=1}^{2}(\rho(2^{-i_l}k_l)\rho(2^{-j_l}k_l))|\psi(\eps k_1)|^2|k_2|^{-4}|\psi(\eps k_2)-1|^2
$$
and remarking that this sum is restricted to the frequency $|k|\lesssim |k_2|\sim|k_1|$ we get easily :
$$
\sum_{k_{12}=k,|k|\lesssim|k_1|\sim|k_2|}\frac{|\psi(\eps k_2)-1|^2|\psi(\eps k_1)|^2}{|k_2|^4}\lesssim |k|^{-2+\delta}r(\eps)
$$
for all $\delta>0$ small enough, with 
$$
r(\eps):=\sum_{k_2\ne0}|k_2|^{-2-\delta}|\psi(\eps k_2)-1|^2\to^{\eps}0
$$ 
by dominate convergence. Then putting this last bound in the definition of $J_1^\eps$ allows to obtain the following inequality : 
$$
J^\eps_1\lesssim r(\eps)2^{2q\delta}.
$$  
To finish our argument let us observe that :
$$
J_2^\eps=\sum_{\substack{q\lesssim i_1\sim i_2\sim j_1\sim j_2\\k\in\mathbb Z^2,k_{12=k}}}|\rho(2^{-q}k)|^2\Pi_{l=1}^{2}(\rho(2^{-i_l}k_l)\rho(2^{-j_l}k_l))|\psi(\eps k_1)||\psi(\eps k_2)-1||\psi(\eps k_2)|\psi(\eps k_1)-1||k_2|^{-2}|k_1|^{-2}
$$
and then due to the fact that the sum is over the frequency $|k_1|\sim|k_2|$ we see that $J_1^\eps\sim J_2^\eps$ from which we can conclude the following bound :
$$
\mathbb E[|\Delta_q((\xi^\eps\circ K\xi-b_\eps)-(\xi^\eps\circ K\xi^\eps-c_\eps))(x)|^2]\lesssim r(\eps)2^{2q\delta}
$$ 
for all $\delta>0$ and $\rho<\delta$ and then using the Gaussian hypercontractivity (see~\cite{janson}) and the Besov embedding we get :
\begin{equation*}
\begin{split}
&||(\xi^\eps\circ K\xi-b_\eps)-(\xi^\eps\circ K\xi^\eps-c_\eps)||^p_{L^p(\Omega,\mathcal C^{-2\delta-2/p})}
\\&\lesssim||(\xi^\eps\circ K\xi-b_\eps)-(\xi^\eps\circ K\xi^\eps-c_\eps)||^p_{L^p(\Omega, B_{p,p}^{-2\delta})}
\\&\lesssim\sum_{q\geq-1}2^{-2qp\delta}\int_{\mathbb T^2}\mathbb E\left[|\Delta_q((\xi^\eps\circ K\xi-b^\eps)-(\xi^\eps\circ K\xi^\eps-c^\eps))(x)|^2\right]^{\frac{p}{2}}\dd x
\lesssim r(\eps)^{p/2}
\end{split}
\end{equation*}
which finishes the proof of the lemma.
\end{proof}
Now we are able to prove the Proposition~\ref{prop:supp-point}.
\begin{lemma}\label{lemma:con-cons}
Let us define $\xi^n$ and $c_n$  
$$
\xi^n:=\sum_{|k|\leq\nu2^n}\hat\xi(k)e_k,\quad c_n=\sum_{|k|\leq\nu 2^n}\frac{1}{|k|^2}
$$
for $\nu>0$. Then for $\nu$ large enough (depending only on the annulus and the ball given in the Littlewood-Paley decomposition) the following convergence 
$$
\lim_{n\to+\infty}T_{-\xi^n+X^{n,c_n-a}}\Xi^{pam}=(0,-a)
$$
holds in $\mathscr H^\alpha$ in probability. Where $X^{n,c_n-a}$ is given by the Lemma~\ref{lemma:pure-area}. Then to obtain the statement of the Proposition~\ref{prop:supp-point} is suffice to take $\Xi=\Xi^{pam}(\omega)$ and $h^k=\xi^{n_k}(\omega)-X^{n_k,c_{n_k}-a}$ with $\omega$ is fixed in the set of probability one for which the last convergence hold  along a subsequence $T_{-\xi^{n_k}+X^{n_k,c_{n_k}-a}}\Xi^{pam}$. 
\end{lemma} 
\begin{proof}
We have by definition that 
$$
\left (T_{-\xi^n+X^{n,c_n-a}}\Xi^{pam}\right )^1=\xi-\xi^n+X^{n,c_n-a}.
$$
To prove that the right hand side of this equality converge to 0 is suffice to remark that $||\xi^n-\xi||_{\alpha-2}\to^{n\to+\infty}0$ and then it suffice to prove that 
$$
||X^{n,c_n-a}||_{\alpha-2}\to^{n\to+\infty}0
$$
but following the proof of the Lemma \ref{lemma:pure-area} we see that 
$$
||X^{n,c_n-a}||_{\alpha-2}\lesssim (c_n)^{1/2}\max(2^{-n(1-\alpha)},\chi(2^nk)).
$$
Then recall in that 
$$
c_n=\sum_{|k|\leq \nu2^n k\in\mathbb Z^2}|k|^{-2}\lesssim_{\nu}n
$$
we deduce easily that 
$$
||X^{n,c_n-a}||_{\alpha-2}\lesssim n^{1/2}\max(2^{-n(1-\alpha)},2^n\chi(2^nk))\to0.
$$
Which is gives the needed convergence for the first component of $T_{-\xi^n+X^{n,c_n}}\Xi^{pam}$. Now by definition we have that 
\begin{equation*}
\begin{split}
(T_{-\xi^n+X^{n,c_n-a}}\Xi^{pam})^2&= \Xi^{pam,2}+\xi^n\circ K\xi^n-\xi^n\circ K\xi-\xi\circ K\xi^n+(\xi-\xi^n)\circ  KX^{n,c_n-a}
\\&+X^{n,c_n-a}\circ K(\xi-\xi^n)+X^{n,c_n-a}\circ  KX^{n,c_n-a}.
\end{split}
\end{equation*}
And let us remark that 
$$
\text{supp}(\mathscr F(\xi^n-\xi))\subseteq\{|k|>\nu2^n\}
$$ 
and that $$
\text{supp}(\mathscr F(X^{n,c_n-a}))\subseteq\{|k|=2^n|z|\}
$$
and then we can choose $\nu$ large enough (depending only on the size of the annulus and the ball which given in the definition of $\chi$ and $\rho$) such that  
$$
\Delta_i(\xi^n-\xi)\Delta_j(X^{n,c_n-a})=0
$$ 
for $|i-j|\leq1$. And then we get immediately that 
$$
X^{n,c_n-a}\circ K(\xi-\xi^n)=(\xi-\xi^n)\circ X^{n,c_n-a}=0
$$
for all $n$. Then we see that 
\begin{equation*}
\begin{split}
(T_{-\xi^n+X^{n,c_n-a}}\Xi^{pam})^2&=\Xi^{pam,2}+(\xi^n\circ K\xi^n-c_n)+(c_n-\xi^n\circ K\xi)
\\&+(c_n-\xi\circ K\xi^n)+(X^{n,c_n-a}\circ  KX^{n,c_n-a}-(c_n-a))-a.
\end{split}
\end{equation*}
Now using the Lemma~\ref{lemma:mixed-conv} we can see that 
$$
||(\xi^n\circ K\xi^n-c_n)+(c_n-\xi^n\circ K\xi)+(c_n-\xi\circ K\xi^n)||_{2\alpha-2}\to^{n\to+\infty}0
$$
in probability.  To obtain the needed convergence is suffice to show that 
$$
||X^{n,c_n-a}\circ  KX^{n,c_n-a}-(c_n-a)||_{2\alpha-2}\to^{n\to+\infty}0.
$$
Once again following the argument given in the Lemma~\ref{lemma:pure-area} we see easily that 
$$
||X^{n,c_n}\circ  KX^{n,c_n}-(c_n-a)||_{2\alpha-2}\lesssim n\max(2^{-2n(1-\alpha)},\chi(2^n|z|))\to^{n\to+\infty}0
$$
we have used the fact that $c_n\lesssim n$. This completes the proof.
\end{proof}
Take $\Xi\in\text{supp}((\Xi^{pam})_\star\mathbb P)$ and $h^k\in\mathcal H$ such that :
$$
T_{-h^k}\Xi\to^{k\to+\infty}(0,-a)\quad  \text{in }\mathscr H^\alpha
$$ 
(this is possible thanks the Proposition~\ref{prop:supp-point}). Moreover we know that the support of $\Xi^{pam}$ is invariant by translation, then $T_{h^k}\Xi^{pam}\in\text{supp}((\Xi^{pam})_\star\mathbb P)$ for all $k$ which give us that $(0,-a)\in\text{supp}((\Xi^{pam})_\star\mathbb P)$. Once again the invariance by translation give us that 
$$
\mathscr X^\alpha\subseteq\text{supp}((\Xi^{pam})^\star\mathbb P)
$$  
which finishes the proof of the Theorem \ref{th:supp-pam}. 
\\
Before going into the proof of the Theorem~\ref{th:main-result} let us observe the fact that the constant $c$ can't dropped from the space $\mathscr X^\alpha$, indeed we claim that 
\begin{lemma}
\label{ref:strict}
Given $\alpha\in(2/3,1)$, then the closure of the set 
$$
\{(h,h\circ Kh),\quad h\in\mathcal H\}
$$
in the space $\mathscr H^\alpha$ is strictly embedded in $\mathscr X^\alpha$.
\end{lemma}
\begin{proof}
Let assume that there exist $h_n$ in $\mathcal H$ such that $(h_n, h_n\circ Kh_n)$ converge in $\mathscr H^\alpha$ to $(0,-1)$. Then the point is that now 
$$
\Delta(Kh_n)^2=2|\nabla Kh_n|^2-2h_nKh_n
$$
with $\nabla$ is the gradient operator. Then using the fact that $h_n\to^{n\to+\infty} 0$ in $\mathscr C^{\alpha-2}$ and thus $\Delta(Kh_n)^2\to^{n\to+\infty}0$ in $\mathscr C^{\alpha-2}(\mathbb T^2)$. On the other side using the Bony estimates~\eqref{proposition:Bony-estim} and the fact that $h_n\circ Kh_n\to^{n\to+\infty}-1$ we obtain easily that : 
$$
h_nKh_n=h_n\prec Kh_n+h_n\succ Kh_n+h_n\circ Kh_n\to^{n\to+\infty}-1
$$ 
in the space $\mathscr C^{\alpha-2}(\mathbb T^2)$. Which allow us to conclude that $2|\nabla Kh_n|^2\to-1$ in the space $\mathscr C^{\alpha-2}(\mathbb T^2)$ which is of course impossible and thus a such sequence can't exist which end the proof due to the fact $(0,-1)\in\mathscr X^\alpha$.

\end{proof}

\section{Proof of the Theorem~\ref{th:main-result}}
We know by construction that $u=\mathscr S_{r}(u_0,\Xi^{pam})$ and that  $\Xi^{pam}\in\mathscr X^\alpha$ a.s. then we can conclude that there exist $\theta^n\in \mathcal H$ and such that $\Xi^{pam}=\lim_{n}(\theta^n,\theta^n\circ K\theta^n-c_n)$ which by the continuity of the map $\mathscr S_r$ give 
$$ 
u=\lim_{n}\mathscr S_r(u_0,(\theta^n,\theta^n\circ K\theta^n-c_n))=\lim\mathscr S_c(u_0,\theta^n,c_n)
$$
a.s in $C([0,T],\mathcal C^\alpha(\mathbb T^2))$. We then have that 
$$
\text{supp}(u_\star\mathbb P)\subseteq\overline{\left\{\mathscr S(u_0,h,c),\quad h\in \mathcal H,c>0\right\}}^{C([0,T],\mathcal C^{\alpha}(\mathbb T^2))}
$$
The other inclusion is more interesting. Now, 
$$
\text{supp}((\Xi^{pam})_\star\mathbb P)=\overline{\left\{(\theta,\theta\circ K\theta-c)\quad \theta\in \mathcal H,c>0 \right\}}^{\mathscr H^\alpha} 
$$ 
ensures that for any $\eta>0$, $c>0$ and $\theta\in L^2(\mathbb T^2)$  we have :
$$
\mathbb P\left(||\Xi^{pam}-\mathscr M(\theta,c)||_{\mathscr H^\alpha}<\eta\right)>0.
$$
Let $\delta>0$ then by the continuity of $\mathscr S_{r}$ there exist $\eta:=\eta(\delta,\theta,c)>0$ such that $||\Xi^{pam}-\mathscr M(\theta,c)||_{\mathscr H^\alpha}\leq\eta\Rightarrow||u-\mathscr S(u_0,\theta,c)||_{C([0,T],\mathcal C^\alpha(\mathbb T^2))}\leq\delta$ and then 
$$
\mathbb P\left(||u-\mathscr S(u_0,h,c)||_{C([0,T],\mathcal C^\alpha(\mathbb T^2))}\leq\delta\right)\geq\mathbb P\left(||\Xi^{pam}-\mathscr M(\theta,c)||_{\mathscr H^\alpha}<\eta\right)>0.
$$
Which proves the first identity set of Theorem~\ref{th:main-result}. At this point let us observe that
$$
u[a]=\mathscr S_r(u^0,\Xi^{pam}+(0,-a)).
$$
Since the support of the law of $\Xi^{pam}+(0,-a)$ is equal to the support of the law of $\Xi^{pam}$ the identity~\eqref{invariance} follows.

\section{Appendix}

\begin{proposition}\label{prop:class}
Let $T>0$. Given $f\in C^3_b(\mathbb R)$, $u_0\in L^2(\mathbb T^2)$ and $h\in \mathcal H$ then there exists a unique global solution $v\in C([0,T]; L^{2}(\mathbb T^2))$ to the equation :
$$
\mathscr Lv=f(v)h-cf'(v)f(v),\quad u(0,x)=u_0(x)
$$
Moreover the map $h\mapsto v$ is continuous from $\mathcal H$ to $C(\mathbb R^+,L^2(\mathbb T^2))$
\end{proposition}
\begin{proof}
Let $a,b\in L^{2}(\mathbb T^2)$ then by a direct computation we get :
$$
|\mathscr{F}((f(a)-f(b))h)(k)|=\left|\sum_{k_1+k_2=k}\mathscr F(f(a)-f(b))(k_1)\mathscr F(h)(k_2)\right|\lesssim||f(a)-f(b)||_{L^2(\mathbb T^2)}||h||_{\mathcal H}
$$
moreover we have that $||f(a)-f(b)||_{L^{2}}\lesssim ||f'||_{L^{\infty}(\mathbb R)}||a-b||_{L^{2}(\mathbb T)}$. And then 
$$
||(f(a)-f(b))h||_{H^\gamma}\lesssim ||f'||_{L^{\infty}(\mathbb R)}||h||_{\mathcal H}||a-b||_{L^{2}(\mathbb T^2)}
$$
for all $\gamma<-1$. Now is suffice to remark that if  $h\in C([0,T], H^\gamma(\mathbb T^2))$ and denoting by $\mathcal P_t=e^{t\Delta}$ the heat flow then the following bound : 
$$
\left|\left|\int_0^t\mathcal P_{t-s}h_s\dd s\right|\right|_{L^2(\mathbb T^2)}\lesssim\int_0^t||\mathcal P_{t-s}h_s||_{L^2}\dd s\lesssim\int_0^t(t-s)^{\gamma/2}\dd s||h||_{C([0,T],H^\gamma(\mathbb T^2))}\lesssim T^{1+\gamma/2}||h||_{C([0,T],H^\gamma(\mathbb T^2))}
$$
hold for $\gamma>-2$. Introducing the map $\Gamma:C([0,T],L^2(\mathbb T^2))\to C([0,T],L^2(\mathbb T^2))$ defined by 
$$
\Gamma_T(v)=P_tu_0+\int_0^t\mathcal P_{t-s}f(v_s)h\dd s-c\int_0^t\mathcal P_{t-s}f'(v_s)f(v_s)\dd s
$$
Due to the last computation this map is well defined moreover it satisfy the following bound :  
\begin{equation*}
\begin{split}
&||\Gamma_T(u)-\Gamma_T(v)||_{C([0,T],L^2(\mathbb T^2))}
\\&\lesssim_c T^{1+\gamma/2}(1+||f||_{L^\infty(\mathbb R)}+||f'||_{L^{\infty}(\mathbb R)}+||f''||_{L^\infty(\mathbb R)})^2
 ||h||_{\mathcal H}||u-v||_{C([0,T],L^2(\mathbb T^2))}
\end{split}
\end{equation*}
for $T<1$ and some $\gamma\in(-2,-1)$. Then choosing $T^\star$ small enough we can see that $\Gamma$ become a contraction on $C([0,T],L^2(\mathbb T^2))$ into itself and then it admit a unique fix point $v^h$. Due to the fact that $T^\star$ does not depend on $||u_0||_{L^2}$ we can iterate our result to obtain a global solution. Now we will us focus on the continuity of the map $h\mapsto v^h$ and let $h_1,h_2\in\mathcal H$ and $R>0$ such that $||h_1||+||h_2||\leq R$ then we have by definition that :
$$
v^{h_1}-v^{h_2}=\int_0^t\dd s\mathcal P_{t-s}f(v^{h_1})(h_1-h_2)+\int_0^t\dd s\mathcal P_{t-s}(f(v^{h_2})-f(v^{h_1}))h_2+c\int_0^t\dd s\mathcal P_{t-s}(g(v^{h_1})-g(v^{h_2}))
$$
with $g=f'f$. Due to the estimate used to stand the fixed point argument we easily get that : 
\begin{equation*}
\begin{split}
&||v^{h_1}-v^{h_2}||_{C([0,T],L^2(\mathbb T^2))}
\\&\lesssim T||f||_{L^\infty(\mathbb T^2)}||h_2-h_1||_{\mathcal H}+T^{\gamma/2+1}(R||f'||_{L^\infty(\mathbb R)}+c||g'||_{L^\infty(\mathbb R)})||v^{h_1}-v^{h_2}||_{C([0,T],L^2(\mathbb T^2))}
\end{split}
\end{equation*}
then choosing $T_1>0$ small enough such that $T_1^{\gamma/2+1}(R||f'||_{L^\infty(\mathbb R)}+c||g'||_{L^\infty(\mathbb R)})<1/2$ allow to obtain that 
$$
||v^{h_1}-v^{h_2}||_{C([0,T_1],L^2(\mathbb T^2))}\lesssim T_1||f||_{L^\infty(\mathbb T^2)}||h_2-h_1||_{\mathcal H}
$$
Now iterating this procedure allow to get finally that 
$$
||v^{h_1}-v^{h_2}||_{C([0,T],L^2(\mathbb T^2))}\lesssim_{R,f} ||h_2-h_1||_{\mathcal H}
$$
for every $T>0$, which end the proof.
\end{proof}
Recall that $\check.$ indicates zero-mean of elements in the appropriate function spaces. 
\begin{proposition}
\label{prop:laplacian-inverse}
Let $T>0$, the map :
$$
\theta \in \check{H}^\alpha  \mapsto -\Delta \theta  \in \check{H}^{\alpha-2}
$$
is invertible. In particular, its inverse 
$$
K:  \check{H^\alpha} \to \check{H}^{\alpha+2}
$$
is well-defined and is a continuous linear operator. The same statement holds if we replace the space $\check{H}^\alpha(\mathbb T^d)$ by $\check{\mathcal C}^\alpha(\mathbb T^d)$.
\end{proposition}
\begin{proof}
For $f\in\mathcal S'(\mathbb T^2)$ with $\hat f(0)=0$ the equation 
$$
-\Delta\theta=f,\quad\hat\theta(0)=0
$$
admit a unique solution $\theta\in\mathscr S'(\mathbb T^2)$ defined by $\hat\theta(k)= |k|^{-2}\hat f(k)$ for $k\ne0$ and $\hat\theta(0)=0$. Moreover by a direct computation we see that  if $f\in\check{H}^{\alpha}$ then $||\theta||_{H^{\alpha+2}}=||f||_{H^\alpha}$ which gives the statement for the Sobolev space. Now if $f\in\check{\mathcal C}^\alpha(\mathbb T^d)$ we have by a direct application of the Proposition~\ref{prop:multiply} that $||\theta||_{\alpha+2}\lesssim||f||_{\alpha}$ and this finishes the proof.
\end{proof}
Now the following lemma ensure that the constant $c$ can't be set it to zero without cost.
\begin{lemma}
\label{lemma:non}
Let $\alpha<1$, $f$ the identity function and $u^0\equiv1$ then in this case the closure of the set 
$$
\left\{\mathscr S(u^0,h,0),\quad h\in\mathcal H\right\}
$$
in the space $C([0,T],\mathcal C^\alpha(\mathbb T^2))$ is strictly contained in the support of the law of $u$ characterized in the Theorem~\ref{th:main-result}.
\end{lemma}
\begin{proof}
Let $v$ the unique solution of the equation 
$$
\mathscr L v=cv,\quad v(0,x)=1
$$
for some fixed $c<0$. Of course $v$ have the explicit formula $v(t,x)=e^{ct}$. Now let $v_n$ a sequence of function which converge to $v$ in $C([0,T],\mathcal C^\alpha(\mathbb T^2))$ and such that  
\begin{equation*}
\label{eq:app}
\mathscr Lv_n=v_nh_n,\quad v_n(0,x)=1
\end{equation*}
for some $h_n\in \mathcal H$. By the Feynman-Kac formula we have immediately that
$$
v_n(t,x)=\mathbb E\left[e^{\int_0^th_n(B_s+x)\dd s}\right]>0
$$
with $B$ is a Brownian motion. Then if we set $\tilde v_n=\log v_n$ we can see that $\tilde v$ satisfy the following equation 
$$
\mathscr L\tilde v_n=|\nabla\tilde v_n|^2+h_n, ,\quad \tilde v_n(0,x)=0.
$$
By passing to the integral on $[0,T]\times\mathbb T^2$ in this last equation and observing that $\int_{\mathbb T^2}h_n=0$ we get   
$$
\int_{\mathbb T^2} \tilde v_n(T,x) dx =\int_0^T\int_{\mathbb T^2}|\nabla\tilde v_n(t,x)|^2\dd x\dd t\geq0.
$$
Now the point is that $\tilde v_n(T,x)$ converges uniformly in $x$ (actually in $\mathcal C^\alpha$) to $\log v(T,x) \equiv cT<0$ and then we conclude that $v$ cannot be approximated by a sequence which satisfies the equation~\eqref{eq:app} which ends the proof.
\end{proof}

\end{document}